\documentclass[11pt,reqno]{amsart}

\usepackage{amsmath}
\usepackage{amsthm}
\usepackage{amssymb}
\usepackage{amsfonts}
\usepackage{latexsym}
\usepackage{hyperref}
\usepackage{cite}
\usepackage{xcolor}
\usepackage{thmtools}

\usepackage{todonotes}
\usepackage{comment}
\usepackage{multirow}
\usepackage{fullpage}
\usepackage{bbm}
\usepackage{algpseudocode}

\newcommand{\imod}[1]{\,(\textnormal{mod }#1)\,}



\newtheorem{theorem}{Theorem}[section]

\newtheorem{lemma}[theorem]{Lemma}

\newtheorem{proposition}[theorem]{Proposition}

\theoremstyle{definition}

\newtheorem*{remark}{Remark}

\newcommand{\logb}[1]{\log{\left(#1\right)}}

\title{On the sum of a prime and a number that is not square-free}

\author[E.~S.~Lee]{Ethan~Simpson~Lee}
\address{University of the West of England, School of Computing and Creative Technologies, Coldharbour Lane, Bristol, BS16 1QY} 
\email{ethan.lee@uwe.ac.uk}
\urladdr{\url{https://sites.google.com/view/ethansleemath/home}}

\author[R.~O'Clarey]{Rowan~O'Clarey}
\address{Universität Hamburg, Department of Mathematics, Bundesstrasse 55, 20146 Hamburg} 
\email{rowan.o.clarey@uni-hamburg.de}

\begin{document}

\begin{abstract}
We prove that every sufficiently large integer $n$ can be written as the sum of a prime and an integer that is not square-free. In addition, we expect this result holds for every $n > 24$ and prove two results to support this claim. First, we prove the result holds unconditionally for every odd $n > 24$. Second, assuming the Generalised Riemann Hypothesis for Dirichlet $L$-functions, we prove the result holds for every $n > 24$. We also discuss the obstruction which prohibits us from proving the result unconditionally for every $n > 24$.
\end{abstract}

\maketitle

\section{Introduction}

Let $g(n)$ denote the number of representations of $n$ as $p+q$, where $p$ and $q$ are primes such that $p\geq q$. The Goldbach conjecture is equivalent to the assertion that the trivial lower bound $g(n) \geq 0$ cannot hold with equality for any even $n > 2$, and Deshouillers, Granville, Narkiewicz, and Pomerance proved the trivial upper bound for $g(n)$ cannot hold with equality for any $n > 210$ in \cite{DeshouillersGranvilleNarkiewiczPomerance}. In this paper, we investigate a variation of the latter result.

In particular, let $T(n)$ denote the number of representations of $n$ as $p + s$, where $p$ is prime and $s$ is a positive square-free integer. Dudek \cite[Thm.~1]{Dudek} proved that the trivial lower bound $T(n) \geq 0$ cannot hold with equality for any $n > 2$. On the other hand, \autoref{thm:main3_suff_large} below proves the trivial upper bound for $T(n)$ cannot hold with equality when $n$ is sufficiently large. To prove \autoref{thm:main3_suff_large}, similar techniques to those in Dudek's paper on the complementary problem are employed, except one has to carefully bound the main term which vanishes for values of $n$ with many prime divisors. 

\begin{theorem}\label{thm:main3_suff_large}
Every sufficiently large integer $n$ can be written as the sum of a prime and an integer that is not square-free. 
\end{theorem}

We expect \autoref{thm:main3_suff_large} to hold for every integer $n>24$ and prove two results to support this claim. First, \autoref{thm:main3_odd} below is an explicit analogue of \autoref{thm:main3_suff_large} for odd $n$. An unconditional analogue for all $n>24$ appears out of reach with the present method; this limitation is too technical to describe satisfactorily here, but it is described in detail in Section \ref{sec:main3_sufflarge}.

\begin{theorem}\label{thm:main3_odd}
Every odd integer $n > 24$ can be written as the sum of a prime and an integer that is not square-free. 
\end{theorem}

A natural approach to overcome the limitation that prohibits us from proving \autoref{thm:main3_odd} for even $n > 24$ is to assume the Generalised Riemann Hypothesis (GRH) for Dirichlet $L$-functions. Under this assumption, we prove \autoref{thm:main3} below, which is a conditional explicit analogue of \autoref{thm:main3_suff_large} for every $n > 24$. 
This result conditionally confirms our expectation and proves that the trivial upper bound for $T(n)$ cannot hold with equality for any $n > 24$. 

\begin{theorem}\label{thm:main3}
Suppose the GRH for Dirichlet $L$-functions is true. Every integer $n > 24$ can be written as the sum of a prime and an integer that is not square-free. 
\end{theorem}

Hardy and Littlewood's Conjecture H (see \cite{HardyLittlewood}) postulates that every sufficiently large integer is either a square or the sum of a prime and a square. Since the set of squares is a subset of the set of positive integers that are not square-free, our results also mark a step toward this conjecture.

The remainder of this manuscript is devoted to proving \autoref{thm:main3_suff_large}, \autoref{thm:main3_odd}, and \autoref{thm:main3}; detailed proofs are given in Section \ref{sec:main_theorems}. Along the way, we require two technical ingredients, namely \autoref{prop:condition} and \autoref{prop:IntegralBounds}. Since their proofs are comparatively long and would interrupt the main line of argument, we defer them to Section \ref{sec:prop:condition} and Section \ref{sec:prop:IntegralBounds}, respectively. 

\subsection*{Acknowledgements} 

ESL thanks the Heilbronn Institute for Mathematical Research for their support. In addition, we thank Daniel Johnston, Chris Keyes, Pieter Moree, Timothy Trudgian, anonymous referees, and other colleagues for valuable feedback and insightful discussions.

\section{Proof of Main Theorems}\label{sec:main_theorems}

To begin, we prove \autoref{lem:comp_checks_nsf} below using computations. This result unconditionally verifies \autoref{thm:main3_odd} and \autoref{thm:main3} for every $n \in (24,8\cdot 10^9]$. Hence it suffices to prove \autoref{thm:main3_odd} and \autoref{thm:main3} for $n > 8\cdot 10^9$. The boundary $8\cdot 10^9$ is chosen, because later we apply certain bounds for the prime number theorem in arithmetic progressions which hold in the range $n\geq 8\cdot 10^9$.

\begin{lemma}\label{lem:comp_checks_nsf}
If $24 < n \leq 8\cdot 10^9$, then $n$ can be written as the sum of a prime and a positive integer that is not square-free.
\end{lemma}

\begin{proof}
For each sub-interval $[3,10^7)$ and $[\alpha\cdot 10^7,(\alpha+1)\cdot 10^7)$ with $\alpha\in[1,800]$, we verify the claim by computation as follows. To begin, we precompute two collections of small integers $m > 1$ that are not square-free: a set $\mathcal{S}_1$ containing integers $m \leq 10^5$ and an ordered list $\mathcal{S}_2$ containing integers $m \leq 10^4$. Next, we verify the bulk of integers in the interval in a covering step. That is, for a collection of primes $p$ in the interval, we mark as verified every integer of the form $p+\overline{s}$ with $\overline{s}\in\mathcal{S}_1$. This covering step greatly reduces the number of cases left to verify. For each $n$ that remains, we test whether $n-\overline{s}$ is prime for some $\overline{s}\in\mathcal{S}_2$, and only if this targeted search fails do we invoke a final brute-force check. The Python code implementing this procedure is available at \href{https://github.com/EthanSLee/On-the-sum-of-a-prime-and-a-square-free-number-with-divisibility-conditions/tree/main}{\texttt{this link}}.
\end{proof}

Next, we derive an explicit criterion which reduces the proofs of \autoref{thm:main3_suff_large}, \autoref{thm:main3_odd}, and \autoref{thm:main3} for large $n$ to the verification of one of the inequalities \eqref{eqn:equv_cond_C-1} or \eqref{eqn:equv_cond_C-2}, which are fully explained below. 
To this end, we introduce the weighted counting function
\begin{equation*}
    R(n) = \sum_{p\leq n} \mu^2(n-p) \log{p},
\end{equation*}
where the sum is over primes $p$, and $\mu$ denotes the M\"{o}bius function. 
Using Dirichlet convolution, we have $$\mu^2(n) = \sum_{a^2 \mid n} \mu(a).$$
From this observation, we can write 
\begin{equation*}
    R(n) = \sum\limits_{a\leq \sqrt{n}} \mu(a)\theta(n,a^{2},n) ,
    \quad\text{where}\quad 
    \theta(n,a^{2},n) = \sum_{\substack{p\leq n\\ p\equiv n\imod{a^2}}} \log{p} .
\end{equation*}
An integer $n$ can be written as the sum of a prime and an integer that is not square-free if and only if 
\begin{equation*}
    R(n) < \theta(n) := \sum_{p\leq n} \log{p} ,
\end{equation*}
which is equivalent to
\begin{equation}\label{eqn:equv_cond_B}
    \sum\limits_{1 < a \leq \sqrt{n}} \mu(a)\theta(n,a^{2},n) < 0 .
\end{equation}
Thus, to prove each of our main results for sufficiently large $n$, it suffices to verify \eqref{eqn:equv_cond_B}. 

The following result gives a practical condition that certifies when \eqref{eqn:equv_cond_B} holds for sufficiently large $n$. 
While this result appears technical, this is balanced by the fact that it leads to a clean, versatile, and verifiable criterion.

\begin{proposition}\label{prop:condition}
Let $0 < A < 1/2$, $c > 1$ be real parameters and 
\begin{equation*}
    P(n) = \prod_{p\nmid n} \Big(1 - \frac{1}{p(p-1)}\Big) .
\end{equation*}
For every integer $1 < a \leq c$ such that $(a,n) = 1$, suppose also that $\varphi$ is the Euler totient function, $E(a,n) = o(n)$ is a function satisfying
\begin{align}\label{required form}
    \left|\theta(n,a^2,n) - \frac{n}{\varphi(a^2)}\right| \leq E(a,n) \quad\text{for all}\quad n\geq x_a > 0 ,
\end{align} 
and $x_0 = \max_{1 < a \leq c} x_a$. We have that \eqref{eqn:equv_cond_B} is true for every $n\geq x_0$ such that 
\begin{equation}\label{eqn:equv_cond_C}
    1 > P(n)
    + n^{-1} \sum\limits_{\substack{1 < a \leq c \\ (a,n) = 1}} \mu^2(a) E(a,n)
    + \left(\frac{1+2A}{1-2A}\right) \frac{4}{c - 1} + \frac{3 \log{n}}{n^{A}} .
\end{equation}
\end{proposition}

\begin{proof}
We adapt the arguments in \cite{Dudek, FrancisLee, HathiJohnston} which have been used to derive lower bounds for $R(n)$. Full details of our proof of \autoref{prop:condition} are given in Section \ref{sec:prop:condition}.
\end{proof}

Next, we require the upper bounds for $P(n)$ presented in \autoref{prop:IntegralBounds} below to verify the condition \eqref{eqn:equv_cond_C} in \autoref{prop:condition} holds for sufficiently large $n$. 

\begin{proposition}\label{prop:IntegralBounds}
Let
\begin{equation}\label{eqn:Wn_def}
    W(n) 
    = \frac{1}{2.11 \log{n} \logb{2.11 \log{n}}} \left(1 - \frac{5}{2\logb{2.11 \log{n}}}\right) .
\end{equation}
If $n \geq 3$ is odd, then $P(n) \leq \tfrac{1}{2}$. 
Moreover, if $n \geq e^{28.05} \approx 1.52\cdot 10^{12}$, then
\begin{equation*}
    P(n) < 1 - W(n) + \frac{W(n)^2}{2} .
\end{equation*}
\end{proposition}

\begin{proof}
Full details of our proof of \autoref{prop:IntegralBounds} are given in Section \ref{sec:prop:IntegralBounds}. 
\end{proof}

Now, we are in position to derive the criterion. To this end, assume the same notation as in \autoref{prop:condition}. 
It follows from \autoref{prop:IntegralBounds} that \eqref{eqn:equv_cond_B} holds for every $n \geq \max\{x_0, e^{28.05}\}$ such that
\begin{equation}\label{eqn:equv_cond_C-1}
    W(n) > \frac{W(n)^2}{2}
    + n^{-1} \sum\limits_{\substack{1 < a \leq c \\ (a,n) = 1}} \mu^2(a) E(a,n)
    + \left(\frac{1+2A}{1-2A}\right) \frac{4}{c - 1} + \frac{3 \log{n}}{n^{A}} .
\end{equation}
In addition, if $n\geq x_0$ is odd, then \eqref{eqn:equv_cond_B} holds whenever
\begin{equation}\label{eqn:equv_cond_C-2}
    \frac{1}{2} > n^{-1} \sum\limits_{\substack{1 < a \leq c \\ (a,n) = 1}} \mu^2(a) E(a,n)
    + \left(\frac{1+2A}{1-2A}\right) \frac{4}{c - 1} + \frac{3 \log{n}}{n^{A}} .
\end{equation}
Therefore, to prove each of our main results, it suffices to locate parameters $A$ and $c$, as well as an admissible definition for $E(a,n)$, such that \eqref{eqn:equv_cond_C-1} (respectively, \eqref{eqn:equv_cond_C-2} in the odd case) holds for all sufficiently large $n$. 


Finally, let $q_i$ denote the $i$th prime and $N\#$ be the product of all distinct primes $p \leq N$. In particular, we have $q_{20} = 71$ and $$q_{20}\# = 557\,940\,830\,126\,698\,960\,967\,415\,390.$$ 
We follow the strategy above to prove \autoref{thm:main3_suff_large} for large $n$, \autoref{thm:main3_odd} for odd $n > 8\cdot 10^9$, and \autoref{thm:main3} for $n \geq q_{20}\#$. All that remains is to prove \autoref{thm:main3} for $8\cdot 10^9 < n < q_{20}\#$, which is done by appealing to the latest conditional bounds for the least prime in an arithmetic progression in an elementary argument. 

Following this plan, we prove \autoref{thm:main3_suff_large} in Section \ref{sec:main3_sufflarge}, \autoref{thm:main3_odd} in Section \ref{sec:main3_odd}, and \autoref{thm:main3} in Section \ref{sec:main3}. Important features of each proof, including the reason we cannot extend the proof of \autoref{thm:main3_odd} to cover even $n$, are also noted as they arise.

\subsection{Proof of \autoref{thm:main3_suff_large}}\label{sec:main3_sufflarge}

The Siegel--Walfisz theorem \cite[Chap.~22,~Eqn.~(4)]{Davenport} implies, for sufficiently large $n$, that there exists a constant $\kappa > 0$ such that if $a\leq (\log{n})^2$, then we can assert
\begin{equation}\label{eqn:SiegelWalfisz}
    E(a,n) \ll n \exp\!\left(-\kappa\sqrt{\log{n}}\right) 
\end{equation} 
in \eqref{eqn:equv_cond_C-1}. 
Therefore, asserting $A = 1/4$, $c = (\log{n})^2$, and \eqref{eqn:SiegelWalfisz} in \eqref{eqn:equv_cond_C-1}, we see that \eqref{eqn:equv_cond_B} is true for sufficiently large $n$. 
The choice $A = 1/4$ is arbitrary. 
This completes our proof of \autoref{thm:main3_suff_large}. \qed

\begin{remark}
The Siegel--Walfisz theorem is ineffective due to the potential existence of a Siegel zero, so our proof of \autoref{thm:main3_suff_large} is also ineffective. 
Recall that ineffective means the implied constant exists but cannot be computed; for further details on the Siegel zero see \cite[Chap.~11.2]{MV}. 
To overcome this limitation, one would need to apply effective bounds for $E(a,n)$ instead of \eqref{eqn:SiegelWalfisz}. Without assuming additional hypotheses, the strongest available effective bounds for $E(a,n)$ have the weaker asymptotic form 
\begin{equation}\label{eqn:BennettEtAll}
    E(a,n) \ll \frac{n}{\log{n}} .
\end{equation}
In particular, explicit bounds of this form are applied below in our proof of \autoref{thm:main3_odd} (see \eqref{eqn:BennettEtAl}). 
However, if we had asserted the effective bound \eqref{eqn:BennettEtAll} in \eqref{eqn:equv_cond_C-1}, then we could not choose $c$ such that the condition \eqref{eqn:equv_cond_C-1} holds for any large $n$. 
To see this, note that if $c > \log{n}$, then
\begin{equation*}
    n^{-1} \sum\limits_{\substack{1 < a \leq c \\ (a,n) = 1}} \mu^2(a) E(a,n)
    \ll \frac{c}{\log{n}}
    \neq o(W(n)) ,
\end{equation*}
and if $c \leq \log{n}$, then
\begin{equation*}
    \frac{4}{c-1} > W(n) .
\end{equation*}
Thus, we must be content with the present ineffective proof of \autoref{thm:main3_suff_large}. Indeed, we believe this underscores the challenge of the problem at hand and enhances its mathematical interest. 
\end{remark}

\subsection{Proof of \autoref{thm:main3_odd}}\label{sec:main3_odd}

In light of \autoref{lem:comp_checks_nsf}, it suffices to prove the result for every odd $n > 8\cdot 10^9$. To this end, recall that Bennett \textit{et al.} \cite[Thm.~1.2]{BennettEtAl} tell us if $n\geq 8\cdot 10^9$, $3 \leq a^2 \leq 10^5$, and $(a,n) = 1$, then we can assert
\begin{equation}\label{eqn:BennettEtAl}
    E(a,n) \leq \frac{n}{160\log{n}} 
\end{equation}
in \eqref{eqn:equv_cond_C-2}. 
For every $1 < a \leq 316$, we have $3\leq a^2 \leq 10^5$. Therefore, assert $c = 316$ and \eqref{eqn:BennettEtAl} in \eqref{eqn:equv_cond_C-2} to see that \eqref{eqn:equv_cond_B} holds for every odd $n > 8\cdot 10^9$ such that
\begin{equation*}
    \frac{1}{2} > \frac{316}{160\log{n}} + \left(\frac{1+2A}{1-2A}\right) \frac{4}{315} + \frac{3 \log{n}}{n^{A}} .
\end{equation*}
This lower bound is minimised at $n = 8\cdot 10^9$ when $A = 0.34843$. Therefore, we also assert $A = 0.34843$ and note the resulting inequality holds for every $n \geq 8\cdot 10^9$. This completes our proof of \autoref{thm:main3_odd}. \qed

\begin{remark}
For a sufficiently large $N$, verifying \eqref{eqn:equv_cond_C-2} for odd $n\geq N$ is substantially simpler than verifying \eqref{eqn:equv_cond_C-1} for all $n\geq N$. This is natural, since when $n$ is odd we have a significantly stronger upper bound for the product $P(n)$, as defined in \autoref{prop:condition}. We exploit this phenomenon to prove \autoref{thm:main3_odd}, which settles all odd cases. However, for the reasons given in Section \ref{sec:main3_sufflarge}, no unconditional extension to even $n$ can be obtained by this approach.
\end{remark}

\subsection{Proof of \autoref{thm:main3}}\label{sec:main3}

Suppose the GRH for Dirichlet $L$-functions is true. In light of \autoref{lem:comp_checks_nsf}, it suffices to prove the result for every odd $n > 8\cdot 10^9$. 

Next, we prove the following lemma and apply it in an elementary argument to establish
\autoref{thm:main3} holds for every $8\cdot 10^9 < n < q_{20}\#$. We treat this intermediate range separately, because the analytic approach we use for $n \geq q_{20}\#$ only yields results when $n$ is large. 

\begin{lemma}\label{lem:condition}
If the GRH for Dirichlet $L$-functions is true and there exists a prime $q$ such that $n > 4(q(q-1)\log{q})^2$ and $(n,q^2) = 1$, then the inequality \eqref{eqn:equv_cond_B} is true. 
\end{lemma}

\begin{proof}
If $(n,q^2) = 1$ and the GRH is true, then Lamzouri, Li, and Soundararajan \cite[Cor.~1.2]{LLS} tell us that the least prime $p\equiv n\imod{q^2}$ satisfies $$p\leq (2\varphi(q^2)\log{q})^2 = 4(q(q-1)\log{q})^2.$$ Therefore, if $n > 4(q(q-1)\log{q})^2$ and $(n,q^2) = 1$, then there exists a prime $p\leq 4(q(q-1)\log{q})^2$ such that $n \equiv p\imod{q^2}$. It follows that $q^2 \mid n - p$, so the inequality \eqref{eqn:equv_cond_B} is true.
\end{proof}

Observe that $8\cdot 10^{9} > 4(q(q-1)\log{q})^2$ for every prime $q\leq q_{20}$, and that for every integer $8\cdot 10^{9} \leq n < q_{20}\#$, there is at least one prime $q\leq q_{20}$ such that $q\nmid n$. 
Thus, there is at least one prime $q\leq q_{20}$ such that $(n,q^2) = 1$. Hence, \autoref{lem:condition} implies the inequality \eqref{eqn:equv_cond_B} is true for every $8\cdot 10^{9} \leq n < q_{20}\#$. 
With this, the result is proved for every $n < q_{20}\#$. 

Finally, recall from \cite[Thm.~2.1]{KeliherLee} or \cite[App.~A]{LeePNTPAPs} that if the GRH for Dirichlet $L$-functions is true, $(a,n)=1$, and $a^2\geq 3$, then
\begin{equation}\label{eqn:theta_GRH}
    E(a,n)
    \leq \left(\frac{(\log{n})^2}{8\pi} + \left(\frac{\log{n}}{\pi} + 4\right)\log{a} + 3.43\right) \sqrt{n} 
\end{equation}
is admissible for every $n\geq 2$ in \eqref{eqn:equv_cond_C-1}. Therefore, upon asserting $c = n^A$ and \eqref{eqn:theta_GRH} in \eqref{eqn:equv_cond_C-1}, we see that \eqref{eqn:equv_cond_B} holds for every $n \geq q_{20}\#$ such that
\begin{equation*}
    W(n) > \frac{W(n)^2}{2}
    + n^{A - \frac{1}{2}} \left(\frac{1 + 8A}{8\pi} (\log{n})^2 + 4 A \log{n} + 3.43\right)
    + \left(\frac{1+2A}{1-2A}\right) \frac{4}{n^A - 1} + \frac{3 \log{n}}{n^A} .
\end{equation*}
This lower bound is minimised at $n = q_{20}\#$ when $A = 0.2419$. Therefore, we also assert $A = 0.2419$ and note the resulting inequality holds for every $n \geq q_{20}\#$. 
With this, \autoref{thm:main3_odd} is proved for every $n > 24$. \qed

\section{Proof of \autoref{prop:condition}}\label{sec:prop:condition}

In this section we establish \autoref{prop:condition} by adapting the methods of \cite{Dudek, FrancisLee, HathiJohnston}. Although certain bounds below admit small refinements, we deliberately give slightly weaker inequalities in inconsequential terms to simplify the final expression. 

Recall the assumptions from the statement of \autoref{prop:condition} and let 
\begin{equation*}
     \Sigma_A(n) = \sum\limits_{\substack{1 < a \leq n^A\\(a,n)=1}} \mu(a)\theta(n,a^{2},n) .
\end{equation*}
First, observe that
\begin{align*}
    \Bigg| \sum\limits_{1 < a \leq n^{1/2}} \mu(a) \theta(n,a^{2},n) - \Sigma_A(n) \bigg|
    &\leq \Big(\sum_{\substack{a \leq n^{A}\\(a,n) > 1}} + \sum\limits_{n^A < a\leq n^{1/2}}\Big) \mu^2(a) \theta(n,a^{2},n) .
\end{align*}
If $(a,n)>1$, then $\theta(n,a^{2},n) \leq \log{n}$, so this bound becomes
\begin{align*}
    \Bigg| \sum\limits_{1 < a \leq n^{1/2}} \mu(a) \theta(n,a^{2},n) - \Sigma_A(n) \bigg|
    &\leq n^{A}\log{n} + \sum\limits_{n^A < a\leq n^{1/2}} \mu^2(a) \theta(n,a^{2},n) .
\end{align*}
Next, follow the arguments to bound $\Sigma_3$ in \cite{Dudek} to see 
\begin{equation*}
    \sum\limits_{n^A < a\leq n^{1/2}} \mu^2(a) \theta(n,a^{2},n)
    \leq \left(n^{1-2A} + n^{1 - A}\right) \log{n} .
\end{equation*}
It follows from these observations that
\begin{equation*}
    \Bigg| \sum\limits_{1 < a \leq \sqrt{n}} \mu(a)\theta(n,a^{2},n) - \Sigma_A(n) \Bigg| 
    \leq 3 n^{1-A}\log{n} .
\end{equation*}
Next, recall that a consequence of the Brun--Titchmarsh theorem (see \cite{MontgomeryVaughan}) is that if $a\leq n^A$ and $(a,n) = 1$, then 
\begin{equation}\label{eqn:BT_application}
    \left|\theta(n,a^2,n) - \frac{n}{\varphi(a^2)}\right| \leq \left(\frac{1+2A}{1-2A}\right) \frac{n}{\varphi(a^{2})} . 
\end{equation}
Apply \eqref{required form} when $1 < a \leq c$ and \eqref{eqn:BT_application} when $c < a \leq n^A$ in the definition of $\Sigma_A(n)$ to reveal that if $n\geq x_0$, then
\begin{equation*}
    \Bigg| \Sigma_A(n) - n \sum\limits_{\substack{1 < a \leq n^A\\(a,n)=1}} \frac{\mu(a)}{\varphi(a^2)} \Bigg| 
    \leq \sum\limits_{\substack{1 < a \leq c\\(a,n)=1}} \mu^2(a) E(a,n)
    + n \left(\frac{1+2A}{1-2A}\right) \sum\limits_{\substack{c < a \leq n^A\\(a,n)=1}} \frac{\mu^2(a)}{\varphi(a^2)} .
\end{equation*}
A straightforward consequence of this bound is that
\begin{equation*}
    \Bigg| \Sigma_A(n) - n \sum\limits_{\substack{a > 1\\(a,n)=1}} \frac{\mu(a)}{\varphi(a^2)} \Bigg| 
    \leq \sum\limits_{\substack{1 < a \leq c\\(a,n)=1}} \mu^2(a) E(a,n)
    + n \left(\frac{1+2A}{1-2A}\right) \sum\limits_{\substack{a > c\\(a,n)=1}} \frac{\mu^2(a)}{\varphi(a^2)} .
\end{equation*}
It follows from \cite[Lem.~3.10]{Ramare} that
\begin{equation*}
    \sum\limits_{a > c} \frac{\mu^2(a)}{\varphi(a^{2})} 
    \leq \frac{4}{c - 1} .
\end{equation*}
Combine this observation with the equality
\begin{equation*}
    \sum\limits_{\substack{a > 1\\(a,n)=1}} \frac{\mu(a)}{\varphi(a^2)}
    = P(n) - 1 
\end{equation*}
to reveal that if $n\geq x_0$, then
\begin{equation*}
\begin{split}
    \Bigg| \sum\limits_{1 < a \leq \sqrt{n}} \mu(a)\theta(n,a^{2},n) &- n(P(n) - 1) \Bigg| \\ 
    &\quad\leq \sum\limits_{\substack{1 < a \leq c\\(a,n)=1}} \mu^2(a) E(a,n)
    + \left(\frac{1+2A}{1-2A}\right) \frac{4 n}{c-1}
    + 3 n^{1-A}\log{n} .
\end{split}
\end{equation*}
It follows from this approximation that \eqref{eqn:equv_cond_B} is true for every $n \geq x_0$ such that \eqref{eqn:equv_cond_C} is true, which completes the proof of \autoref{prop:condition}. \qed

\section{Proof of \autoref{prop:IntegralBounds}}\label{sec:prop:IntegralBounds}

In this section, we prove \autoref{prop:IntegralBounds}. To begin, note that if $n \geq 3$ is odd, then $2\nmid n$, so $P(n) \leq 1/2$ follows naturally. 

All that remains is to prove the general upper bound, which requires a more involved argument. To this end, let $\omega(k)$ denote the number of distinct prime factors of $k$, $q_i$ denote the $i$th prime, and note that
\begin{equation}\label{eqn:the_engine}
    P(n)
    = \prod_{p\nmid n} \left( 1 - \frac{1}{p(p-1)}\right)
    \leq C_{\text{Artin}} \prod_{p \leq q_{\omega(n)}} \left( 1 + \frac{1}{p^2 - p -1}\right) ,
\end{equation}
where $$C_{\text{Artin}} = \prod_{p} \left( 1 - \frac{1}{p(p-1)}\right) = 0.3739558136\ldots$$ is Artin's constant. To complete our proof of \autoref{prop:IntegralBounds}, we prove the result when $\omega(n) > 2$ and $\omega(n) \leq 2$ independently.

\subsection{Case I} 
Suppose $\omega(n) > 2$. To bound the product in \eqref{eqn:the_engine}, we prove an explicit upper bound for $q_{\omega(n)}$ of the form $q_{\omega(n)} \ll \log{n}$ in the following lemma. 

\begin{lemma}\label{lem:Robin_consequence}
If $n \geq 55$ and $\omega(n) > 2$, then $q_{\omega(n)} < 2.11 \log{n}$.
\end{lemma}

\begin{proof}
Recall that if $n \geq 3$, then Robin \cite[Thm.~11]{Robin} tells us 
\begin{equation}\label{eqn:RI}
    \omega(n) \leq \frac{1.3841 \log{n}}{\log\log{n}} .
\end{equation}
If $\omega(n) \geq 6$, then \cite[(3.13)]{RosserSchoenfeld} implies
\begin{align*}
    q_{\omega(n)} 
    < \left(1 + \frac{\log\log{\omega(n)}}{\log{\omega(n)}}\right) \omega(n) \log{\omega(n)} 
    &< \left(1 + \frac{1}{e}\right) \omega(n) \log{\omega(n)} \\
    &< \frac{10}{9} \left(1 + \frac{1}{e}\right) \omega(n) \log{\omega(n)} . 
\end{align*}
It is straightforward to verify this inequality also holds in the extended range $\omega(n) > 2$. 
Hence, if $n\geq 55$ and $\omega(n) > 2$, then \eqref{eqn:RI} tells us
\begin{align*}
    q_{\omega(n)} 
    &< \frac{10}{9} \left(1 + \frac{1}{e}\right) \frac{1.3841 \log{n}}{\log\log{n}} \logb{\frac{1.3841 \log{n}}{\log\log{n}}} \\
    &< \frac{2.11 \logb{1.3841 \log{n} / \log\log{n}}}{\log\log{n}} \log{n} 
    < 2.11 \log{n} , 
\end{align*}
since $\logb{1.3841 \log{n}/\log\log{n}}/\log\log{n}$ decreases from $0.99790\dots$ to $0.73421\dots$ on the interval $$(54.12154, e^{43.0505\dots}]$$ and increases to one on $(e^{43.0505\dots},\infty)$.
\end{proof}

It follows from \eqref{eqn:the_engine} and \autoref{lem:Robin_consequence} that if $n\geq 55$, then
\begin{align*}
    P(n)
    \leq C_{\text{Artin}} \prod_{p \leq 2.11 \log{n}} \left( 1 + \frac{1}{p^2 - p -1}\right) 
    &= \prod_{p > 2.11 \log{n}} \left( 1 - \frac{1}{p(p-1)}\right) \\
    &< \prod_{p > 2.11 \log{n}} \left( 1 - \frac{1}{p^2}\right) .
\end{align*}
Now, if $y \geq 59$, then \cite[Thm.~1]{RosserSchoenfeld} tells us
\begin{equation}\label{eqn:pi_bounds}
    \frac{y}{2(\log{y})^2} < \pi(y) - \frac{y}{\log{y}} < \frac{3y}{2(\log{y})^2} ,
\end{equation}
where $\pi(y)$ is the prime-counting function. 
It follows from partial summation and integration by parts that
\begin{align*}
    \sum_{p > y} \frac{1}{p^{2}}
    &= - \frac{\pi(y)}{y^{2}} + 2 \int_{y}^{\infty} \frac{\pi(t)}{t^{3}}\,dt \\
    &= - \frac{1}{y \log{y}} 
    + 2 \int_{y}^{\infty} \frac{dt}{t^{2}\log{t}}
    - \frac{\pi(y) - \tfrac{y}{\log{y}}}{y^{2}} 
    + 2 \int_{y}^{\infty} \frac{\pi(t) - \tfrac{t}{\log{t}}}{t^{3}}\,dt \\
    &= \frac{1}{y \log{y}} 
    - 2 \int_{y}^{\infty} \frac{dt}{t^{2}(\log{t})^2}
    - \frac{\pi(y) - \tfrac{y}{\log{y}}}{y^{2}} 
    + 2 \int_{y}^{\infty} \frac{\pi(t) - \tfrac{t}{\log{t}}}{t^{3}}\,dt .
\end{align*}
Therefore, it follows from \eqref{eqn:pi_bounds} that if $y \geq 59$, then 
\begin{align*}
    \sum_{p > y} \frac{1}{p^{2}}
    &> \frac{1}{y \log{y}} 
    - \frac{3/2}{y (\log{y})^2}
    - \int_{y}^{\infty} \frac{dt}{t^{2} (\log{t})^2} \\
    &\geq \frac{1}{y \log{y}} \left( 1 - \frac{5}{2\log{y}} \right) .
\end{align*}
With this, we have proved that if $y \geq 59$, then 
\begin{align*}
    \prod_{p > y} \left( 1 - \frac{1}{p^2}\right)
    &\leq \exp\left( - \sum_{p > y} \frac{1}{p^{2}}\right) 
    \leq \exp\left( - \frac{1}{y \log{y}} \left( 1 - \frac{5}{2\log{y}} \right) \right) .
\end{align*}
Let $y = 2.11 \log{n}$, then $y \geq 59$ for all $n \geq e^{28.05}$, so we restrict our attention to such $n$. It follows from this and $e^{-x} \leq 1 - x + x^2 / 2$, which holds for any $x > 0$, that
\begin{align*}
    P(n)
    &< e^{- W(n)}
    \leq 1 - W(n) + \frac{W(n)^2}{2} ,
\end{align*}
where $W(n)$ has been defined in \eqref{eqn:Wn_def}, which is the desired upper bound. 

\subsection{Case II} 
Suppose $\omega(n) \leq 2$. In this case, \eqref{eqn:the_engine} implies 
\begin{align*}
    P(n)
    &\leq C_{\text{Artin}} \prod_{p \leq 3} \left( 1 + \frac{1}{p^2 - p -1}\right) 
    = \frac{12}{5} C_{\text{Artin}} = 0.89749\dots ,
\end{align*}
which is stronger than the desired upper bound. \qed

\bibliographystyle{amsplain}
\bibliography{refs}
\end{document}